\documentclass[11pt]{article}
\usepackage{amsthm}
\usepackage[mathscr]{eucal}
\usepackage{graphics}
\usepackage{endnotes}

\usepackage{epsfig}

\newtheorem{theorem}{Theorem}[section]
\newtheorem{lemma}{Lemma}[section]
\newtheorem{proposition}{Proposition}[section]
\newtheorem{definition}{Definition}[section]
\newtheorem{corollary}{Corollary}[section]

\newtheorem{remark}{Remark}

\input{epsf}

\title{Ground subgroups}
\author{U.A. Rozikov\textsuperscript{1}\\
{\small \textsuperscript{1} Institute of Mathematics and
Information Technologies,\hfill{\ }}\\
\ \ {\small 29, Do'rmon Yo'li str., Tashkent, 100125, Uzbekistan.}\\ \ \ \ \ \ {\small Email: rozikovu@yandex.ru \hfill {\ }}}

\date{}

\begin{document}

\maketitle

\begin{abstract}
In this article we give a concept of ground subgroup
for finite and countable groups.
By our definition such a subgroup of a group depends on a given subset
of the group and on a given partition of the subset.
For finite and free groups we describe some sets of ground subgroups.
We apply the ground subgroups to describe ground states of a model
of statistical mechanics.

\end{abstract}

\section{Introduction} There are several thousand papers
and books devoted to the theory of groups. But still there are
unsolved problems, most of which arise
in solving of problems of natural sciences as physics,
biology etc. In particular, if configuration of a physical
system is located on a lattice (in our case on the graph of a group) then the configuration
can be considered as a function defined on the lattice. Usually,
more important configurations (functions) are periodic ones.
It is well-known that if
the lattice has a group representation then periodicity of a
function can be defined by a given subgroup of the representation.
More precisely, if a subgroup, say $H$,  is given, then one can
define $H$-periodic function as a function, which has a constant value
(depending only on the coset)
on each (right or left) coset of $H$. So the periodicity is related to a special partition
of the group (that presents the lattice on which our
physical system is located). There are many works devoted to
several kind of partitions of groups (lattices) (see e.g. \cite{CL},\cite{GR}, \cite{KM}-\cite{R2},
\cite{Y}). 

In this paper we study more general problem: given a subset $A$ of a
finite or countable group $G$, and given a partition of $A$, is there any subgroup $H$
of $G$ such that cosets of which divide the set $A$ exactly as the given partition?
This problem arises if, for example, one want to study the periodic ground configurations (states)
of a physical system with the Hamiltonian (energy) which is a sum of
interaction functions $I_A$ defined on $A$ (see e.g. \cite{G}, \cite{R3}, \cite{S}).
Note that a Hamiltonian is a function of configurations. A {\it ground state} of a Hamiltonian is a
(periodic) configuration which minimizes the Hamiltonian. Thus subgroups which we want to describe
are related to the ground states of a Hamiltonian, so we call them {\it ground subgroups}.

The paper is organized as follows. In section 2 we give all
necessary definitions and formulation of the problem.
In sections 3-5 we describe ground subgroups of
finite and free groups. A ground subgroup 
depends on a given subset, say $A$ and on a given
partition of the subset $A$. For finite groups we show that one can
choose $A$ and its partition such that there does not exist any ground subgroup
corresponding to this partition of the subset.
 Section 6 is devoted to an application of results to description of ground states of
a Hamiltonian defined on a Cayley tree. In the last section we discuss the results 
and give some open problems.

\section{Definitions and statement of the problem}

Let $G$ be a finite or countable group and $H$ be a subgroup of $G$. For $x\in G$
denote $Hx=\{yx: y\in H\}$ the right coset of the subgroup $H$.
Define the relation $\sim$ on $G$ of right congruence by $a\sim b$
if and only if $ab^{-1}\in H$. This relation is equivalence on $G$.
Hence the right cosets of $H$ are pairwise disjoint.
The cardinal of the collection of all right cosets is called the index
of the subgroup $H$ in the group $G$ and is denoted by $|G:H|$.
Let $H\subset G$ is a subgroup with index $r\in N\cup\{+\infty\}$.
Denote by $H_1=H, H_2,...,H_r$ the right cosets of $H$. For an
arbitrary finite subset (not necessary subgroup) $A\subset G$ denote
\begin{equation}
q^H_{A,i}=|A\cap H_i|, \ i=1,...,r, \label{1}
\end{equation}
where $|S|$ denotes the cardinal of $S$.

\begin{definition}\label{d1}
Let $A$ be a subset of $G$ and $\overrightarrow{k}=(k_1,...,k_m)$ be a vector
with
\begin{equation}
k_1,...,k_m\in \{1,...,|A|\},\ \ k_1\leq k_2\leq ...\leq k_m \ \ \mbox{and} \ \ \sum^m_{i=1}k_i=|A|.\label{2}
\end{equation}
A subgroup $H$ of index $r\geq m$ is called $(A,\overrightarrow{k})$-{\it ground subgroup} if there are
$i_1,...,i_m\in \{1,...,r\}, \ i_p\ne i_q, \ p\ne q$ such that $q^H_{A,i_j}=k_j, j=1,..,m.$
\end{definition}

Denote by ${\cal H}(A,\overrightarrow{k})\equiv {\cal H}_G(A,\overrightarrow{k})$ the set of all $(A,\overrightarrow{k})-$ ground
subgroups of a group $G$. For ${\cal H}(A,\overrightarrow{k})\ne \emptyset$ we denote
\begin{equation}
r_0\equiv r_0(A,\overrightarrow{k})=\min\{|G:H| : H\in {\cal H}(A,\overrightarrow{k})\}.\label{3}
\end{equation}

\begin{remark}
1. The ground subgroups can be similarly defined for the left cosets
of $H$.

2. If $A\subset G$ and a subgroup $H$ of index $r\geq 1$ are given
then it is obvious that the subgroup $H$ is $(A,\overrightarrow{k})$-ground subgroup
with $\overrightarrow{k}=(k_1,...,k_m)$ such that $k_j=|A\cap H_j|\ne 0$, $j=1,...,m.$
Thus ${\cal H}(A)=\bigcup_{\overrightarrow{k}}{\cal H}(A,\overrightarrow{k})$ is the set of all subgroups of $G$.
Moreover if we know all subgroups of a group $G$ then for a given $(A,\overrightarrow{k})$
it will be easy to describe the set ${\cal H}(A,\overrightarrow{k})$. But the inverse problem:
for a given $A$ and $\overrightarrow{k}$ with condition (\ref{2}) finding of an
$(A,k)$-ground subgroup is not easy.
\end{remark}

So the main problem of this paper is

{\bf Problem 1.} Let $G$ be a finite or countable group. For a given $A\subset G$ and $\overrightarrow{k}$
with conditions (\ref{2}) describe ${\cal H}(A,\overrightarrow{k})$ and find $r_0=r_0(A,\overrightarrow{k})$.

\begin{remark}
 Since $|A\cap H_i|= |gA\cap H_i|$ for any $g\in H$ and $i=1,..,r$, we have that
 if $H$ is an $(A,\overrightarrow{k})$-ground subgroup then it is $(gA, \overrightarrow{k})$-ground subgroup for any $g\in H$.
\end{remark}

\section{Simple groups}

Recall that in every group the set containing only the identity element 1,
and the group itself, are subgroups. Subgroups other than these are
called proper subgroups. A group is called simple if it has no proper
subgroups.

One can easily prove the following
\begin{proposition}\label{p1}
If $G$ is a simple group then for all $A\subset G$

1) ${\cal H}(A,\overrightarrow{k})=\left\{\{1\}\right\}$ and $r_0=|G|$
if $\overrightarrow{k}=(k_1,...,k_m)$ with $k_i=1$,
for all $i=1,...,m$.

2) ${\cal H}(A,\overrightarrow{k})=\left\{G\right\}$ and $r_0=1$
if $\overrightarrow{k}=(|A|)$ i.e. $m=1$.

3) ${\cal H}(A,\overrightarrow{k})=\emptyset$  if
$\overrightarrow{k}=(k_1,...,k_m)$ with at least one $k_i\in \{2,...,|A|-1\}$.
\end{proposition}

\section{Cyclic groups}

Let $G$ be a cyclic group of order $n$ ($n$ may be infinity).
It is known that every subgroup of $G$ is cyclic. Moreover,
the order of any subgroup of $G$ is a divisor of $n$ and
for each positive divisor $p$ of $n$ the group $G$ has
exactly one subgroup of order $p$. Each infinite subgroup of $G$ is
$pZ$ for some $p$, which is bijective to (so isomorphic to) $Z$.
All factor groups of $Z$ are finite, except for the trivial
exception  $Z/{0}=Z/0Z$.

The following theorem gives upper estimation
for the set ${\cal H}$.

\begin{theorem}\label{t1}
Let $G$ be a cyclic group of finite order $n$ and
$1=n_0<n_1<...<n_{p-1}<n_p=n$ are all divisors of $n$.
Let $H^{(n_i)}$ be the subgroup of $G$ with index $n_i, i=0,...,p$.
For $A\subset G$  and $\overrightarrow{k}$ with conditions (\ref{2}) we have
$${\cal H}(A,\overrightarrow{k})\subseteq \widehat{{\cal H}}(A,\overrightarrow{k})=\left\{H^{(n_i)}: m\leq n_i\leq {n\over \|\overrightarrow{k}\|}\right\},$$
where $\|\overrightarrow{k}\|=\max\{k_i: 1\leq i\leq m\}$.
\end{theorem}
\begin{proof} By construction of $\overrightarrow{k}$ we have $m\leq n_i$.
Let $H^{(n_i)}_j, j=1,2,..., {n\over n_i}$
be the right cosets of $H^{(n_i)}$. We have
$$
q^{H^{(n_i)}}_{A,j}=|A\cap H^{(n_i)}_j|=k_j\leq {n\over n_i},$$
i.e. $n_i\leq {n\over k_j}$ for all $j=1,..,m$ which implies
$n_i\leq {n\over \|\overrightarrow{k}\|}$.
\end{proof}
\begin{corollary}
 If $m>{n\over \|\overrightarrow{k}\|}$ then ${\cal H}(A,\overrightarrow{k})=\emptyset$.
 \end{corollary}
 The following example shows that the estimation of the set ${\cal H}(A,\overrightarrow{k})$
 given in Theorem \ref{t1}
 can not be improved i.e there are some $A$ and $\overrightarrow{k}$  such that
 ${\cal H}(A,\overrightarrow{k})=\widehat{{\cal H}}(A,\overrightarrow{k})$ and even there are
 some $A$ and $\overrightarrow{k}$ such that ${\cal H}(A,\overrightarrow{k})=\emptyset$.

 {\bf Example 1.} Consider $G=\{0,1,2,3,4,5\}, n=6$.
 All subgroups of $G$ are
 $$H^{(1)}=G, \ H^{(2)}=\{0,2,4\}, \ H^{(3)}=\{0,3\},\ H^{(6)}=\{0\}.$$
 Consider $A=A_1=\{1,2,3\}$ then $\overrightarrow{k}$ can be one of the following
 vectors:
 $$\overrightarrow{k}^{(1)}=(1,1,1),\ \overrightarrow{k}^{(2)}=(1,2),\ \overrightarrow{k}^{(3)}=(3)$$
and we have
$${\cal H}(A,\overrightarrow{k}^{(1)})=\left\{H^{(3)}, H^{(6)}\right\}=\widehat{{\cal H}}(A,\overrightarrow{k}^{(1)}),$$
$$ {\cal H}(A,\overrightarrow{k}^{(2)})=\left\{H^{(2)}\right\}\subset \widehat{{\cal H}}(A,\overrightarrow{k}^{(2)}), \
{\cal H}(A,\overrightarrow{k}^{(3)})=\left\{H^{(1)}\right\}\subset \widehat{{\cal H}}(A,\overrightarrow{k}^{(3)}).$$
Hence if $A=A_1$ for all $\overrightarrow{k}$ we have non-empty set of
ground subgroups. Now consider the case $A=A_2=\{0,1,2,4\}$ then $\overrightarrow{k}$ can be one of the following
 vectors:
 $$\overrightarrow{k}^{(1)}=(1,1,1,1),\ \overrightarrow{k}^{(2)}=(1,1,2),\ \overrightarrow{k}^{(3)}=(1,3),\ \overrightarrow{k}^{(4)}=(2,2), \overrightarrow{k}^{(5)}=(4)$$
and we obtain
$${\cal H}(A,\overrightarrow{k}^{(1)})=\left\{H^{(6)}\right\}\subset \widehat{{\cal H}}(A,\overrightarrow{k}^{(1)}),$$
$$ {\cal H}(A,\overrightarrow{k}^{(2)})=\left\{H^{(3)}\right\}=\widehat{{\cal H}}(A,\overrightarrow{k}^{(2)}), \
{\cal H}(A,\overrightarrow{k}^{(3)})=\left\{H^{(2)}\right\}=\widehat{{\cal H}}(A,\overrightarrow{k}^{(3)}),$$
$$ {\cal H}(A,\overrightarrow{k}^{(4)})=\emptyset, \
{\cal H}(A,\overrightarrow{k}^{(5)})=\left\{H^{(1)}\right\}=\widehat{{\cal H}}(A,\overrightarrow{k}^{(5)}).$$
So for $A=A_2$ we can have all possible cases: empty set, subset and equality.

Now we shall consider the case $n=\infty$ i.e $G=Z$.
Let $A$ be a finite subset of $Z$ and $p\in N$. Denote
$$A_{p,i}=\{x\in A: x=i({\rm mod} p)\}, \ i=0,1,...,p-1.$$
The following proposition is obvious

\begin{proposition}\label{p3} If $A\subset Z$ and $\overrightarrow{k}=(k_1,...,k_m)$
with (\ref{2}) are given, then
$pZ\in {\cal H}(A,\overrightarrow{k})$ if and only if
for any $j\in \{1,...,m\}$ there is  $i_j\in \{0,1,...,p-1\}$ such that
$|A_{p,{i_j}}|=k_j$.
\end{proposition}

Points $x,y\in Z$ are called nearest-neighbors if $|x-y|=1$.

\begin{proposition} Let $A$ be an arbitrary finite subset of $Z$.
 Assume that there exists $q\in \{2,...,|A|\}$ such that any $B\subset A$ with
$|B|=q$ has at least one pair of nearest-neighbors. Then
${\cal H}(A,\overrightarrow{k})=\emptyset$ for any $\overrightarrow{k}=(k_1,...,k_m)$ with $m>1$ and
at least one coordinate $k_i=q$.
\end{proposition}
\begin{proof} For $m>1$, it is easy to see that $Z\notin {\cal H}(A,\overrightarrow{k})$.
Assume that there is $p>1$ such that $pZ\in {\cal H}(A,\overrightarrow{k})$
then by definition we should have $|A\cap (pZ+r)|=q$ for
some $r\in\{0,1,...,p-1\}$. Denote $B=A\cap (pZ+r)$, by conditions of
theorem we have $x,y\in B$ such that $x=y+1$. Since  $x, y\in pZ+r$,
$x-y=0({\rm mod} p)$ i.e. $1=0({\rm mod} p)$ and $p=1$ this is contradiction
to $p>1$.
\end{proof}
\section{Free groups}
Let $G$ be a countable free group.
For $A\subset G$ we denote
\begin{equation}
A^*=\{z\in G: \exists x,y\in A, z=xy^{-1}\}.\label{4}
\end{equation}
  If $G$ is generated by a set $M$, then $M$ is called
an irreducible set of generators if no proper subset of  $M$ is a set
of generators for $G$.
\begin{lemma}
1) $1\in A^*$, where $1$ is the identity of $G$. Moreover $A^*=\{1\}$ iff $|A|=1$;
2) $A=A^*$ iff $A$ is a subgroup;

3) $A\cap A^*=\emptyset$ iff $A$ is a subset of an irreducible
set of generators of $G$;

4) $A\subset A^*$ iff $1\in A$.
\end{lemma}
\begin{proof} 1)-3) are straightforward.

4) If $1\in A$ then for any  $x\in A$ we have $x1^{-1}=x\in A^*$ i.e $A\subset A^*$.
Now assume $A\subset A^*$ and $1\notin A$ then since $A$ is a
set as $A=\{x_1,...,x_q\}$ we should have
$1\ne x_m=x_ix_j^{-1}$ for any $m=1,...,q$ and some $i=i(m), j=j(m)\in\{1,...,q\}, i\ne j$
i.e. $x_mx_jx_i^{-1}=1$ this is contradiction to the assumption that $G$ is a free group.
\end{proof}

Note that if $\mathbf{1}=(1,...,1)$ is the vector defined by conditions (\ref{2})
then for an arbitrary $A\subset G$ we have $\{1\}\in {\cal H}(A, \mathbf{1})$.
But $|G:\{1\}|=\infty$. Now we shall give a construction
of an $(A,\mathbf{1})$-ground subgroup with finite index.

We shall use the following
\begin{theorem}\label{t2} {\rm \cite{K}} If $x\in G\setminus\{1\}$,
then there exists a normal subgroup $H_x$ of $G$, of
finite index, that does not contain $x$.
\end{theorem}
Denote
\begin{equation}
H_{A}=\bigcap_{x\in A^*\setminus\{1\}}H_x,\label{5}
\end{equation}
where $A\subset G$ with  $|A|>1$,  $A^*$ is defined
by (\ref{4}) and $H_x$ is given in Theorem \ref{t2}.

Denote by $\alpha(x)$ the number of these
generators that occur in the reduced form of $x\in G$ and by $|x|$
the length of $x$.
\begin{proposition}\label{pp} The normal subgroup $H_A\subset G$
is a $(A, \mathbf{1})$-ground subgroup
for any $A\subset G$, with $|A|>1$. Moreover
\begin{equation}
\prod_{x\in A^*\setminus \{1\}}(\alpha(x)+1)\leq
|G:H_A|\leq \prod_{x\in A^*\setminus \{1\}}(|x|+1)!. \label{6}
\end{equation}
\end{proposition}
\begin{proof} Let $H_{A,j}$ be a right coset of $H_A$.
By our construction of $H_A$ we get $H_A\cap A^*=\{1\}$,
consequently, for any $x,y\in A$ we have $xy^{-1}\notin H_A$, i.e
$|A\cap H_{A,j}|\in \{0,1\}$, for any $j$.
Following the proof of Theorem \ref{t2}, given in \cite{K} page 42,
one can easily see that
\begin{equation}
\alpha(x)+1\leq
|G:H_x|\leq (|x|+1)!, \label{7}
\end{equation}
which by (\ref{5}) gives (6).
\end{proof}

\section{An application: a model on a Cayley tree}

In this section we consider a model of statistical mechanics on a Cayley tree.
The Cayley tree $\Gamma^k=(V,L)$ of order $ k\geq 1 $ is an infinite tree,
i.e., a graph without cycles,
from each vertex of which exactly $k+1$ edges issue. Here $V$ is the set of
vertices and $L$ is the set of edges of $\Gamma^k$.

It is known that there exists a one-to-one
correspondence between the set  $V$ of vertices  of the Cayley tree
of order $k\geq 1$ and the group $G_{k}$ of the free products of
$k+1$ cyclic groups $\{e, a_i\}$, $i=1,...,k+1$ of the second order
(i.e. $a^2_i=e$, $a^{-1}_i=a_i$) with generators $a_1, a_2,...,
a_{k+1}$.

For $A\subseteq V$ a
spin {\it configuration} $\sigma_A$ on $A$ is defined as a function
 $x\in A\to\sigma_A(x)\in\Phi=\{1,2,...,q\}$; the set of all configurations
 coincides with $\Omega_A=\Phi^{A}$. We denote $\Omega=\Omega_V$ and $\sigma=\sigma_V$.
 Define a {\it periodic configuration} as a configuration $\sigma\in \Omega$
which is invariant under a subgroup of shifts $G^*_k\subset G_k$ of
finite index.

In \cite{R3} the following model of statistical mechanics was considered. 
For $A\subset V$ define
 the function $U: \sigma_A\in\Omega_A\to
U(\sigma_A)\in \big\{|A|-1, |A|-2,..., |A|-\min\{|A|, |\Phi|\}\big\}$ by
 \begin{equation} U(\sigma_A)=|A|-|\sigma_A\cap \Phi|, \label{8}
 \end{equation}
  where  $\Phi=\{1,2,...,q\}$ and $|\sigma_A\cap \Phi|$ is the
 number of distinct values of $\sigma_A(x), x\in A.$
 For instance if $\sigma_A$ is a constant configuration then $|\sigma_A\cap \Phi|=1$.

 Note that if $|A|=2$, say  $A=\{x,y\}$, then $U(\{\sigma(x),\sigma(y)\})=\delta_{\sigma(x),\sigma(y)}$
 i.e $U$ is a generalization of the Kronecker symbol.

Denote by $M_r$ the set of all balls
$b_r(x)=\{y\in V: d(x,y)\leq r\}$ with radius $r\geq 1$, where $d(x,y)$ is the distance
on the Cayley tree i.e the number of edges of the shortest path connecting $x$ and $y$.

 Now consider the Hamiltonian
\begin{equation}
 \mathbf{H}(\sigma)=-J\sum_{b\in M_r}U(\sigma_b), \label{9}
\end{equation}
 where $J\in R\setminus\{0\}.$

\begin{remark}
Since $U$ is a generalization of the Kronecker symbol, the 
Hamiltonian (\ref{9}) is a natural generalization of Potts model, 
for details about Potts model see e.g. \cite{GR}, \cite{S}.
\end{remark}
A periodic configuration $\phi\in \Omega$ is called
 a {\it ground state} of the Hamiltonian (\ref{9}) if
 $\mathbf{H}(\phi)\leq \mathbf{H}(\sigma)$ for all $\sigma\in \Omega$.

  Denote by  $GS(\mathbf{H})$ the set of all ground states of $\mathbf{H}$. Put
  $$K\equiv K(k, r)=|b_r(x)|=1+(k-1)^{-1}(k+1)(k^r-1),$$
  where $k\geq 2$ is the order of the Cayley tree and $r\geq 1$ 
  is an integer number.

\begin{theorem}\label{t3} 1. If $J>0$ then $GS(\mathbf{H})$ contains the constant configurations only
i.e it contains $q$ configurations $\sigma^{(i)}\equiv i$, $i=1,..,q$.

2. If $q$ is large enough and $J<0$ then $GS(\mathbf{H})$ contains at least ${q!\over (q-K)!}$ periodic
ground states, which are periodic with respect to the $(b, \mathbf{1})$-ground
subgroup $H_b$, with $b\in M_r$.
\end{theorem}
 \begin{proof} 1) It is easy to see that for $J>0$ a
 configuration $\phi$ minimizes $\mathbf{H}$ iff
 it maximizes  $U$ on any ball $b\in M_r$ i.e iff it
 is a constant configuration.
 2) For $J<0$ a configuration $\phi$ minimizes $\mathbf{H}$ iff
 it also minimizes $U$ on any ball $b\in M_r$. So on any ball $b\in M_r$
 the configuration has to be with distinct values.  First consider $b(e)\in M_r$
 i.e the ball with the center $e$.  Note that an analogue of
   Theorem \ref{t2} is true for the group $G_k$ (see \cite{GR}). Consequently,
      Proposition \ref{pp} is also true for $G_k$. Hence we have a
      $(b(e), \mathbf{1})$-ground subgroup $H$ of $G_k$. It is easy to see that
      this subgroup is also $(b(g),\mathbf{1})$ -ground subgroup for any
      $g\in G_k$, where $b(g)\in M_r$ is a ball with the center $g$. So now we can define
      a $H$-periodic ground state $\phi$ as a function $\phi:G_k\to \Phi$ such that 
      $\phi(x)=i$ if $x$ is an element of
      the (right) coset $H_i$ of $H$. Since $q$ is large enough we can always choose
      distinct values of the configuration $\phi$ on distinct cosets and
      number of such choices is equal to ${q\choose K}\cdot K!={q!\over (q-K)!}$.\end{proof}
\begin{remark} By the proof 2) of Theorem \ref{t3} it is clear that if we know $|G_k:H|$
then $q$ can be chosen as $q\geq |G_k:H|$.
\end{remark}
\section{Discussion and open problems} For a given finite or countable group $G$ we have
defined a concept of ground subgroup for a subset $A$, and its given partition.
If groups are finite, we obtained a set of such subgroups and showed that this set
  can be empty for some suitable choice of $A$ and its partition. For free groups
   we proved that for an arbitrary $A$ one can construct a subgroup which divides
   the set $A$ to distinct (non-equivalent, see section 2) elements.
   In the section 6 we applied this result to describe the ground states of a
   model with an arbitrary finite interaction radius (i.e diameter of the balls: $2r$).
One of the key problems related to the (spin) models is the description of the set
of Gibbs measures. This problem has a good connection with the problem of
the description the set of ground states. Because the phase diagram of Gibbs
measures is close to the phase diagram of the ground states
for sufficiently small temperatures (see e.g.\cite{G}, \cite{S}).

Results of the paper show that, in general, the Problem 1 (see section 2) 
is very difficult. For example, I do not have any deep result about exact 
value of $r_0$ (see Problem 1). 

Note that from Problem 1 one can get more simple problem:
   let $A\subset G$ and its partition $A_i, \ i=1,...,m$, with $|A_i|=k_i$ are given,
if a subgroup $H$ is $(A, \overrightarrow{k})$ -ground subgroup then we can assume that
 $A_1\subset H$ but
  $A_i\cap H=\emptyset$ for any $i=2,...,m$. So if we denote $B=A\setminus A_1$ then
   the following problem is a particular case of our problem:

 {\bf Problem 2.} Let $G$ is a finite or countable group. Consider arbitrary subsets
 $A, B\in G$ such that $1\in A$ and $A\cap B=\emptyset$.
 Is there any subgroup $H=H(A,B)$ of $G$ such that $A\subset H$ and $B\cap H=\emptyset$?

Note that for an arbitrary  $A\subset G$ one can easily show
that there is $H$ with $A\subset H$. In section 5, for free groups
we have showed that for any $B\subset G$
there exists a subgroup $H_B$ such that $B\cap H_B=\emptyset$ if $1\notin B$. Even for
free groups I have not any idea how to construct a subgroup with the
properties mentioned in the Problem 2.

The following problem is very general:

{\bf Problem 3.} Let $A_1, A_2,...,A_m\subset G$ such
that $A_i\cap A_j=\emptyset$, $i\ne j$. Is there any subgroup
$H$ of $G$, with index $\geq m$, such that $A_i$ is a subset
of a right coset $H_{p(i)}$ and $p(i)\ne p(j)$ for $i\ne j$?

Note that Problem 1 is a particular case of the Problem 3. Indeed
any subgroup $H$ satisfying the conditions of the Problem 3 also satisfies
  the conditions of the Problem 1, with $A=\bigcup_iA_i$.

{\bf Acknowledgements.}
I thank the Abdus Salam International Center for
Theoretical Physics (ICTP), Trieste, Italy for providing financial
support of my visit to ICTP (February-April 2009).

\end{document}